\documentclass[11pt,oneside]{article}

\usepackage{amsthm, amsmath, amssymb, amsfonts,epsfig}
\usepackage{mathtools}
\usepackage{esint}
\usepackage{epstopdf}

\usepackage{graphicx}
\usepackage[font=sl,labelfont=bf]{caption}
\usepackage{subcaption}

\usepackage{enumerate}

\usepackage[colorlinks=true, pdfstartview=FitV, linkcolor=blue,
            citecolor=blue, urlcolor=blue]{hyperref}
\usepackage[usenames]{color}
\definecolor{Red}{rgb}{0.7,0,0.1}
\definecolor{Green}{rgb}{0,0.7,0}

\usepackage[numbers]{natbib}

\usepackage{url}
\makeatletter\def\url@leostyle{%
 \@ifundefined{selectfont}{\def\UrlFont{\sf}}{\def\UrlFont{\scriptsize\ttfamily}}} \makeatother

\usepackage{accents, comment}

\usepackage{bbm}

\usepackage{mathrsfs}

\newtheorem{theorem}{Theorem}

\newtheorem{lemma}[theorem]{Lemma}
\newtheorem{corollary}[theorem]{Corollary}

\theoremstyle{definition}

\theoremstyle{remark}
\newtheorem{remark}[theorem]{Remark}

\numberwithin{equation}{section}
\numberwithin{theorem}{section}

\def\cH{\mathcal{H}}

\def\cM{\mathcal{M}}

\def\bR{\mathbb{R}}

\title{Oscillations and integrability of the vorticity in the 3D NS flows}
\author{Y. Do$^1$ \and A. Farhat$^2$ \and Z. Gruji\'c$^1$ \and L. Xu$^1$}
\date{%
       $^1$University of Virginia, \ $^2$Florida State University
       \\[2ex] %
\today
}

\begin{document}

\maketitle

\begin{abstract}

\bigskip

\noindent In the studies of the Navier-Stokes (NS) regularity problem, it has become increasingly clear that a more realistic path to improved \emph{a priori} bounds is to try to break away from the scaling of the energy-level estimates in the realm of the blow-up-type arguments (the solution in view is regular/smooth up to the possible blow-up time) rather than to try to improve the regularity of Leray's weak solutions. The present article is a contribution in this direction; more precisely, it is shown--in the context of an algebraic/polynomial-type blow-up profile of \emph{arbitrary degree}--that a very weak condition on the \emph{vorticity direction} field (membership in a local $bmo$ space weighted with \emph{arbitrary many} logarithms)--suffices to break the energy-level scaling in the bounds on the vorticity. At the same time, the obtained bounds transform a 3D NS criticality scenario depicted by the macro-scale long vortex filaments into a no-singularity scenario.

\bigskip

\end{abstract}

\section{Introduction}

The 3D Navier-Stokes (NS) equations read
\[
 u_t + (u \cdot \nabla) u =  \nu \triangle u - \nabla p + f,
\]
supplemented with the incompressibility condition $ \ \mbox{div} \ u = 0$. A vector field $u$ denotes the velocity of
the fluid, a scalar field $p$ the pressure, a positive constant $\nu$ the viscosity, and a vector field $f$ the external
force. In what follows, for simplicity, the viscosity is set to $1$, $f$ taken to be a potential force and the spatial domain
to be the whole space $\mathbb{R}^3$.

\medskip

Taking the curl of the $(u, p)$ formulation yields the equations for the vorticity of the fluid $\omega$ ($\omega = \ \mbox{curl} \ u$),
\[
 \omega_t + (u \cdot \nabla) \omega = \triangle \omega + (\omega \cdot \nabla) u.
\]
(Since we are on $\mathbb{R}^3$, the velocity can be recovered from the vorticity by the Biot-Savart law, and one
obtains a closed system for $\omega$.) In this setting, the only piece of data is the initial condition.

\medskip

The left-hand side represents the transport of the vorticity by the velocity, the first term on the right-hand side is the diffusion,
and the second term is the vortex-stretching term. The vortex-stretching term $ (\omega \cdot \nabla) u$ is responsible for
amplification of the vorticity magnitude in 3D NS flows and possible formation of singularities (it is identically zero in 2D).

\medskip

The NS regularity problem is super-critical, i.e., there is a scaling gap between any regularity class/criterion and
the corresponding \emph{a priori} bound (with respect to the unique scaling leaving the NS equations invariant).
Moreover, the scaling gap has been of the fixed size; all the regularity classes are (at best) scaling-invariant, while all the
\emph{a priori} bounds have been on the scaling level of the Leray's energy class $L^\infty(0, T; L^2)$. Somewhat
unexpectedly, a very recent work
(\citet{Bradshaw2017}) introduced a mathematical framework based on the suitably defined scale of spatial
intermittency/sparseness of the regions of intense vorticity--and motivated by the computational simulations of turbulent
flows--in which the regularity class is still scaling invariant, but the \emph{a priori} bound is algebraically/polynomially
better than the energy-level bound (this is in the context of a blow-up-type argument; i.e., one studies an initially smooth
solution in
the vicinity of the first possible blow-up time).

\medskip

There are two main classical approaches in the study of the NS regularity problem; one is to try to improve the regularity of the
Leray's weak solutions (or some other class of weak solutions), and the other one is to study the initially smooth/regular flows
approaching a possible singularity (a blow-up-type argument). Given that there is an increasing evidence that  Leray's
solutions may not be unique (see, for example, a very recent work of \citet{Buckmaster2017}), and the fact that regular solutions are
necessarily unique, the second avenue seems more realistic. In addition, the aforementioned reduction
of the scaling gap (\cite{Bradshaw2017}) took place within the realm of the blow-up-type arguments.

\medskip

The main goal of this article is to show that--in the context of a blow-up-type argument--one can escape the energy-level
scaling of the vorticity \emph{a priori} bounds under a very weak condition on the oscillations of the vorticity direction
field $\xi$. We were motivated by the work of \citet{Bradshaw2015} where it was shown that as long as $\xi$
is uniformly-in-time bounded in the logarithmically weighted space of the functions of bounded mean oscillations
$\displaystyle{\widetilde{bmo}_\frac{1}{|log r|}}$, the vorticity magnitude stays bounded in the Zygmund space
$L \log L$ breaking the energy-level scaling $L^1$. A notable feature of this result is that the space
$\displaystyle{\widetilde{bmo}_\frac{1}{|log r|}}$ allows for discontinuities (e.g., discontinuities of
$\sin \log |\log r|$-type). At the same time, the question arose of whether it was possible to break the energy-level scaling
in the vorticity bounds under a weaker assumption on $\xi$ (allowing for stronger discontinuities), or whether there is a
natural obstruction and the space
$\displaystyle{\widetilde{bmo}_\frac{1}{|log r|}}$ is optimal within this framework.

\medskip

Unfortunately, there is an obstruction. Namely, one of the main ingredients in \cite{Bradshaw2015} is the
Coifman-Rochberg $BMO$-estimate on the logarithm of the maximal function (\cite{Coifman1980}) which has no analogue in the weighted
$BMO$-spaces. In particular, in the case where the weight is a $k$-fold $\log$-composite (the case of the most interest
for applications to the 3D NSE), several types of
counterexamples are presented in the Appendix. It turned out that the obstruction disappears if one restricts the
considerations to the case where the function of interest (its magnitude) exhibits an algebraic/polynomial-type singularity
(there is no restriction on the degree/strength of the singularity). Fortunately, in the realm of the 3D NSE, an
algebraic/polynomial type of the blow-up profile (for the vorticity magnitude) is consistent with the current
NS theory, and--in this context--it is indeed possible to generalize the result from \cite{Bradshaw2015}, i.e., replace
the original $\xi$-condition $\displaystyle{\widetilde{bmo}_\frac{1}{|log r|}}$ with the class of $\xi$-conditions
\[
 \widetilde{bmo}_\frac{1}{\ln\ln\ldots\ln|\ln r|}
\]
where the number of the $\log$-composites can be an arbitrary positive integer, and still obtain the bound on the
vorticity magnitude escaping the
energy-level scaling.

\medskip

In addition, the estimates on the vorticity magnitude obtained here yield improved estimates on the distribution function
of the vorticity sufficient to transform a geometric criticality scenario presented in \citet{Dascaliuc2012} into a regularity
scenario.

\medskip

The paper is organized as follows. Section 2 presents several results on weighted local $BMO$s and maximal functions;
in particular, a `dynamic' generalization of the Coifman-Rochberg estimate in the setting of the algebraic/polynomial-type
singularity is presented. Section 3 contains the main result.

\section{Results on Weighted BMOs and Maximal Functions}

In this section, we list and prove some results in harmonic analysis that will be applied in the next section. Henceforth,
for a set $S$ and a function $f$, $f(S)$ will denote the average value of $f$ on $S$, $f(S)=\fint_s f = \frac{1}{|S|} \int_S f$.

\begin{lemma}\label{le:WBMO-Multiplier}
Let $\varphi$ be a suitable weight function and $\delta>0$ sufficiently small (depending on $\varphi$).
For any $f\in L^1$ define
\begin{align*}
\|f\|_{\widetilde{bmo}_\varphi}:=\|f\|_{L^1}+\underset{|Q|<\delta}{\sup}\ \frac{1}{\varphi(r)}\fint_Q\left|f(x)-f(Q)\right|dx.
\end{align*}
Suppose $\phi: (0,\infty)\to [0,\infty)$ is increasing near 0, and let $\phi_*$ be defined by
\begin{align}\label{eq:phistar}
\phi_*(t)=\int_{\min(\delta, t)}^{\delta}\frac{\phi(s)}{s}ds.
\end{align}
Then
\begin{align}\label{eq:PtwMulti}
\|fg\|_{\widetilde{bmo}}\lesssim \|f\|_{\widetilde{bmo}_\phi}\left(\|g\|_\infty+\|g\|_{bmo_{(\phi_*)^{-1}}}\right).
\end{align}
(Here, $\widetilde{bmo} = \widetilde{bmo}_1$.)
\end{lemma}

\begin{proof}
We slightly modify the proof of Theorem~2 in \citet{Janson1976} as follows. Let $Q=Q(x_0,r)$ be any cube
with center $x_0$ and side length $r<\frac{\delta}{2}$. Then, as shown in \cite{Janson1976},
\begin{align*}
\left|f(Q)\right|\lesssim \|f\|_{\widetilde{bmo}_\phi}\int_{r}^\delta\frac{\phi(t)}{t}dt\lesssim \phi_*(r)\|f\|_{\widetilde{bmo}_\phi}.
\end{align*}
The following string of inequalities is then easily verified.
\begin{align*}
(fg)(Q) &\lesssim \fint_{Q(x_0,r)}|g|\left|f-f(Q)\right|dx + \fint_{Q(x_0,r)}|f|\left|g-g(Q)\right|dx
\\
&\lesssim \|g\|_\infty \fint_{Q(x_0,r)}\left|f-f(Q)\right|dx + |f(Q)|\fint_{Q(x_0,r)}\left|g-g(Q)\right|dx
\\
&\lesssim \phi(r)\|f\|_{\widetilde{bmo}_\phi}\|g\|_\infty +\|f\|_{\widetilde{bmo}_\phi}\phi_*(r)\fint_{Q(x_0,r)}\left|g-g(Q)\right|dx
\\
&\lesssim \|f\|_{\widetilde{bmo}_\phi}\|g\|_\infty + \|f\|_{\widetilde{bmo}_\phi}\|g\|_{bmo_{(\phi_*)^{-1}}}.
\end{align*}
\end{proof}

\bigskip

\begin{lemma}\label{le:Diff-MaxFun}
For an $0<r<1$ and a function $f$ define the following incarnation of the maximal operator:
$M_r(f)=\left(M\left(|f|^r\right)\right)^{1/r}$ where the usual maximal operator is defined as
\begin{align}\label{eq:DefMaxFunc}
M(f)(x)=\underset{x\in Q}{\sup}\frac{1}{|Q|}\int_Q |f(y)|dy
\end{align}
where the supremum is taken over all cubes containing $x$. Let $\{f_t\}$ be a family of functions satisfying
\begin{enumerate}
\item $\{f_t\}\subset C^1(\bR^n\setminus\{0\})$
\item $\underset{t}{\sup}\ |\nabla f_t|\lesssim |x|^{-\ell}$
\item $\underset{t}{\inf}\ |f_t|\gtrsim 1$.
\end{enumerate}
Then there exists $r\in (0,1)$ such that $\nabla M_r(f_t)$ exists almost everywhere and
\begin{align*}
\underset{t}{\sup}\ |\nabla M_r(f_t)|\lesssim |x|^{-k}
\end{align*}
where $k$ depends only on $r$ and $\ell$.
\end{lemma}

\begin{proof}
Let $f$ be a member $\{f_t\}$. Note that $\nabla M_r(f)$ exists almost everywhere because the maximal function maps H\"older continuous function to H\"older continuous function and in particular Lipschitz function to Lipschitz function (cf. \citet{Buckley1999}).
Let $g=|f|^r$. We choose $r$ small enough so that $|g|\lesssim |x|^{-m}$ and $m<n-1$. Fix an $x\in Q$ and let $\tilde{Q}=(y-x)+Q$. Then $y\in\tilde{Q}$ if and only if $x\in Q$, and for any $Q$ containing $x$
\begin{align*}
\fint_Q g dz\le & \fint_Q g dz - \fint_{\tilde{Q}} g dz + \fint_{\tilde{Q}} g dz
\\
\le &\ \underset{y\in\tilde{Q}}{\sup} \left|\fint_Q g dz - \fint_{\tilde{Q}} g dz\right| + \underset{y\in\tilde{Q}}{\sup} \fint_{\tilde{Q}} g dz
\\
=&\ \underset{x\in Q}{\sup} \left|\fint_Q g dz - \fint_{\tilde{Q}} g dz\right| + \underset{y\in Q}{\sup} \fint_{Q} g dz.
\end{align*}
Consequently
\begin{align*}
|M(g)(x)-M(g)(y)|\le \ \underset{x\in Q}{\sup} \left|\fint_Q g dz - \fint_{\tilde{Q}} g dz\right|.
\end{align*}
Now fix an $x\neq0$ and suppose $y$ is close enough to $x$. Let $Q_0$ be the cube centered at 0 with side length $|x|/4$. If $Q\cap (2Q_0)=\emptyset$, then for $y$ close enough to $x$
\begin{align*}
\left|\fint_Q g dz - \fint_{\tilde{Q}} g dz\right|\big/|x-y| \le \fint_Q\frac{\left|g(z+y-x)-g(z)\right|}{|x-y|}dz \le \fint_Q |\nabla g(\xi_{z,x,y})|dz\lesssim |x|^{-\ell}.
\end{align*}
If $Q\cap (2Q_0)\neq\emptyset$, then $\textrm{diam}(Q)\ge |x|/2$ (the diameter of $Q$) and
\begin{align*}
\left|\fint_Q g dz - \fint_{\tilde{Q}} g dz\right|\big/|x-y| \le (|Q||x-y|)^{-1}\left(\int_{(Q\Delta\tilde{Q})\cap Q_0}g dz+\int_{(Q\Delta\tilde{Q})\cap Q_0^c}g dz\right).
\end{align*}
For the first part of the right hand side, we have estimates
\begin{align*}
(|Q||x-y|)^{-1}\int_{(Q\Delta\tilde{Q})\cap Q_0^c}g dz\lesssim &\ |x|^{-1} \fint_{(Q\Delta\tilde{Q})\cap Q_0^c}g dz\lesssim |x|^{-1}|x|^{-(\ell-1)r}\lesssim |x|^{-(\ell-1)r-1}.
\end{align*}
For the second part we have
\begin{align*}
(|Q||x-y|)^{-1}\int_{(Q\Delta\tilde{Q})\cap Q_0}g dz \lesssim &\ |x|^{-1} \fint_{(Q\Delta\tilde{Q})\cap (2Q_0)}g dz\lesssim |x|^{-1}\fint_{(Q\Delta\tilde{Q})\cap (2Q_0)}|z|^{-m} dz.
\end{align*}
Note that the maximal average of $|z|^{-m}$ takes place when $Q\Delta\tilde{Q}$ contains the origin, in which case $(Q\Delta\tilde{Q})\cap (2Q_0)$ is a flat cube inside $2Q_0$ whose height is at most $|x-y|$. The largest possible average is attained when $y\to x$, i.e. the flat cube shrinks to an area in $\bR^{n-1}$. So
\begin{align*}
(|Q||x-y|)^{-1}\int_{(Q\Delta\tilde{Q})\cap Q_0}g dz &\lesssim |x|^{-1}\fint_{\left(2Q_0\right)\cap\bR^{n-1}}|\tilde{z}|^{-m}d\tilde{z}
\\
&\lesssim \ |x|^{-1}|x|^{1-n}|x|^{-m+n-1}\lesssim |x|^{-m-1}.
\end{align*}
To sum up, for some fixed $k$ and for $y$ close enough to $x$, we have
\begin{align*}
\left|\fint_Q g dz - \fint_{\tilde{Q}} g dz\right|\big/|x-y| \lesssim |x|^{-k}.
\end{align*}
Combining the two cases yields
\begin{align*}
|\nabla M(g)(x)|\lesssim\underset{\textrm{ all directions} \ y}{\sup}\ \lim_{y\to x}|M(g)(x)-M(g)(y)|\big/|x-y|\lesssim |x|^{-k_1}.
\end{align*}
Note that the second assumption implies $M(g)\lesssim |x|^{-\ell'}$ for some $\ell'$. Consequently, \begin{align*}
|\nabla M_r(f)(x)|\lesssim |x|^{-k_2}\quad\textrm{for some }k_2,
\end{align*}
finishing the argument.
\end{proof}

\bigskip

\begin{lemma}\label{le:WBMO-tDym}
Suppose a family of differentiable functions $\{f_t\}_{0<t<T}$ is compactly supported in $B(0,1)$ (open ball with radius 1 centered at 0) and
\begin{enumerate}
\item $|f_t(x)|\gtrsim\displaystyle\frac{1}{|x|^{\epsilon_1}+(T-t)^{\epsilon_2}} $ in $B(0,\frac{1}{2})$ for some $\epsilon_1,\epsilon_2>0$
\item $\underset{t}{\sup}\ |\nabla f_t(x)|\lesssim \displaystyle\frac{1}{|x|^{\ell_1}}$ for some $\ell_1>0$
\item $\underset{x}{\sup}\ |\nabla f_t(x)|\lesssim \displaystyle\frac{1}{(T-t)^{\ell_2}}$ for some $\ell_2>0$.
\end{enumerate}
Set
$\phi_*(x)=\ln^{(k)}\left(|\ln x|\right)=\log\log\ldots\log|\log x|$ (for some $k \in \mathbb{N}$)
and $\phi(x)=-x\phi_*'(x)$ for small $x$.
Then, for a suitable $r\in(0,1)$, $\underset{t}{\sup} \ \|\phi_*(M_r f_t)\|_{\widetilde{bmo}_\phi}\lesssim 1$.
\end{lemma}

\begin{proof}
Lemma~\ref{le:Diff-MaxFun} and the second assumption imply
\begin{align}\label{eq:DiffMaxBddx}
\underset{t}{\sup}\ |\nabla M_r f_t|\lesssim \displaystyle\frac{1}{|x|^{\ell_1}}
\end{align}
for some $r\in(0,1)$. In fact $\ell_1$ can be different from the one given in the assumption, but for convenience we still denote it by $\ell_1$. Following a similar argument as in Lemma~\ref{le:Diff-MaxFun}, we deduce that
\begin{align*}
\left|M(|f_t|^r)(x)-M(|f_t|^r)(y)\right|\le \underset{x\in Q}{\sup}\left|\fint_Q |f_t(z)|^rdz- \fint_{\tilde{Q}}|f_t(z)|^rdz\right|
\end{align*}
and that
\begin{align*}
\frac{\left|M(|f_t|^r)(x)-M(|f_t|^r)(y)\right|}{|x-y|}&\le \underset{x\in Q}{\sup}\left|\fint_Q |f_t(z)|^rdz- \fint_{\tilde{Q}}|f_t(z)|^rdz\right|/|x-y|
\\
&\lesssim \underset{x\in Q}{\sup}\fint_Q\frac{\left||f_t|^r(z+y-x)-|f_t|^r(z)\right|}{|x-y|}dz
\\
&\lesssim \underset{x\in Q}{\sup}\fint_Q|\nabla |f_t|^r(\xi_{z,x,y})|dz\ .
\end{align*}
Therefore $|\nabla M(|f_t|^r)(x)|\lesssim \displaystyle\underset{x\in Q}{\sup}\fint_Q|\nabla |f_t|^r(\xi_{z,x,y})|dz$. Having in mind the first assumption implies $\underset{t,x}{\inf}\ |f_t|(x)\gtrsim 1$, we see
\begin{align}\label{eq:MaxDerInTime}
|\nabla M(|f_t|^r)(x)|\lesssim \displaystyle\underset{x\in Q}{\sup}\fint_Q|\nabla |f_t|(\xi_{z,x,y})|dz\ .
\end{align}
The second assumption implies that $\underset{t}{\sup}\ |f_t|\lesssim |x|^{-\ell_1+1}$, so $\underset{x\in\partial B(0,\frac{1}{2})}{\sup}|f_t|\lesssim 1$. Then the third assumption implies that
\begin{align}\label{eq:MaxFuncTimeBdd}
\underset{x\in B(0,\frac{1}{2})}{\sup}M(f_t)\lesssim \frac{1}{(T-t)^{\ell_2}}\ .
\end{align}
Combining \eqref{eq:MaxFuncTimeBdd} and \eqref{eq:MaxDerInTime} with the third assumption we obtain
\begin{align}\label{eq:DiffMaxBddt}
\underset{x}{\sup}\ |\nabla M_r f_t|\lesssim \displaystyle\frac{1}{(T-t)^{\ell_2'}}\quad\textrm{for some }\ell_2'\ .
\end{align}
Now we divide the proof into three cases:

If $\textrm{dist}(0,Q)\gtrsim(\textrm{diam}(Q))^\gamma$ with $\gamma<1/\ell_1$, note that by \eqref{eq:DiffMaxBddx}
\begin{align*}
\fint_Q \phi_*(M_r f_t)dx - \underset{Q}{\inf}\phi_*(M_r f_t)&\lesssim \underset{Q}{\sup}\ \phi_*(M_r f_t)- \underset{Q}{\inf}\phi_*(M_r f_t)
\\
&\lesssim \underset{Q}{\sup}\ |\phi_*'(M_r f_t)\cdot\nabla M_r f_t|\cdot \textrm{diam}(Q)
\\
&\lesssim \phi_*'\left(\underset{Q}{\inf}M_r f_t\right)\cdot\underset{Q}{\sup}\ |\nabla M_r f_t|\cdot \textrm{diam}(Q)\ .
\end{align*}
From the first assumption we know $\underset{Q}{\inf}M_r f_t\gtrsim 1$, so $\phi_*'\left(\underset{Q}{\inf}M_r f_t\right)\lesssim 1$, and
\begin{align*}
\fint_Q \phi_*(M_r f_t)dx - \underset{Q}{\inf}\phi_*(M_r f_t)&\lesssim (\textrm{dist}(0,Q))^{-\ell_1}\cdot\textrm{diam}(Q)
\\
&\lesssim |Q|^{(1-\gamma\ell_1)/n}\lesssim \phi(|Q|)\ .
\end{align*}

If $T-t\gtrsim(\textrm{diam}(Q))^\delta$ with $\delta<1/\ell_2$, similarly, by \eqref{eq:DiffMaxBddt}
\begin{align*}
\fint_Q \phi_*(M_r f_t)dx - \underset{Q}{\inf}\phi_*(M_r f_t)&\lesssim \phi_*'\left(\underset{Q}{\inf}M_r f_t\right)\cdot\underset{Q}{\sup}\ |\nabla M_r f_t|\cdot \textrm{diam}(Q)
\\
&\lesssim (T-t)^{-\ell_2}\textrm{diam}(Q)
\\
&\lesssim |Q|^{(1-\delta\ell_2)/n}\lesssim \phi(|Q|)\ .
\end{align*}
The last inequality follows from the fact that $(1-\delta\ell_2)/n>0$ and $|x|^{(1-\delta\ell_2)/n}\ll \phi(|x|)$ when $|x|$ is sufficiently small.

If $\textrm{dist}(0,Q)\lesssim (\textrm{diam}(Q))^\gamma$ and $T-t\lesssim(\textrm{diam}(Q))^\delta$, we divide the proof into two subcases. Case~(i): if $0\notin Q$, then by the first assumption we have
\begin{align*}
\underset{Q}{\inf}M_r f_t &\ge \underset{Q}{\inf} |f_t|\gtrsim \frac{1}{(\textrm{dist}(0,Q)+\textrm{diam}(Q))^{\epsilon_1}+(T-t)^{\epsilon_2}}
\\
&\gtrsim (\textrm{diam}(Q))^{-\min\{\gamma\epsilon_1,\delta\epsilon_2\}}\gtrsim |Q|^{-\epsilon_3}.
\end{align*}
This implies (by Jensen's inequality and the mean value theorem)
\begin{align*}
\fint_Q \phi_*(M_r f_t)dx - \underset{Q}{\inf}\phi_*(M_r f_t)\lesssim &\ (\phi_*\circ\exp)\left(\fint_Q \ln(M_r f_t)\ dx\right) - (\phi_*\circ\exp)\left(\underset{Q}{\inf}\ln(M_r f_t)\right)
\\
\lesssim &\ (\phi_*\circ\exp)'\left(\xi_Q\right)\left(\fint_Q \ln(M_r f_t)\ dx - \underset{Q}{\inf}\ln(M_r f_t)\right)
\end{align*}
where $\underset{Q}{\inf}\ln(M_r f_t)\le \xi_Q\le \fint_Q \ln(M_r f_t) dx$. Note that $(\phi_*\circ\exp)'$ is decreasing near 0, and that by \citet{Coifman1980} we have $\fint_Q \ln(M_r f_t)\ dx - \underset{Q}{\inf}\ln(M_r f_t)\lesssim 1$; consequently
\begin{align*}
\fint_Q \phi_*(M_r f_t)dx - \underset{Q}{\inf}\phi_*(M_r f_t)\lesssim (\phi_*\circ\exp)'\left(\ln|Q|^{-\epsilon_3}\right)\lesssim \phi(|Q|)\ .
\end{align*}
Case~(ii): if $0\in Q$, then
\begin{align*}
\underset{Q}{\inf}M_r f_t\ge \underset{Q}{\inf} |f_t|\gtrsim \frac{1}{(\textrm{diam}(Q))^{\epsilon_1}+(T-t)^{\epsilon_2}}\gtrsim (\textrm{diam}(Q))^{-\min\{\epsilon_1,\delta\epsilon_2\}}
\end{align*}
which leads back to the proof of Case~(i).
\end{proof}

\section{Main Result}

This section is devoted to our main result which is formulated in the following theorem.

\begin{theorem}\label{th:LogkEst-WBMO}
Let $u$ be a Leray solution to the 3D NSE on $\mathbb{R}^3 \times (0, \infty)$.
Assume that the initial vorticity $\omega_0\in L^1\cap L^2$, and that $T>0$ is the first (possible) blow-up time.
Suppose that
\begin{enumerate}
\item for some choice of the parameters $\epsilon_1,\epsilon_2,\ell_1,\ell_2>0$ and some neighborhood $U$ of 0,
\begin{itemize}
\item $|\omega(x,t)|\gtrsim\displaystyle\frac{1}{|x|^{\epsilon_1}+(T-t)^{\epsilon_2}} $
\item $\underset{t<T}{\sup}\ |\nabla \omega(x,t)|\lesssim \displaystyle\frac{1}{|x|^{\ell_1}}$
\item $\underset{x\in U}{\sup}\ |\nabla \omega(x,t)|\lesssim \displaystyle\frac{1}{(T-t)^{\ell_2}}$
\end{itemize}
(in particular, $\epsilon_1,\epsilon_2$ can be arbitrarily small and $\ell_1,\ell_2$ can be arbitrarily large),
\item setting $\varphi(r)=\left(\ln^{(k)}|\ln r|\right)^{-1}=\left(\ln\ln\ldots\ln|\ln r|\right)^{-1}$ and $\xi(t,x)=\omega(t,x)/|\omega(t,x)|$,
\begin{align}\label{eq:VortexWBMO}
\underset{t\in(0,T)}{\sup}\|(\psi\xi)(t,\cdot)\|_{\widetilde{bmo}_{\varphi}}<\infty.
\end{align}
\end{enumerate}
Then
\begin{align}
\underset{t\in(0,T)}{\sup}\int \psi(x)w(x,t)\Phi(w(x,t))dx<\infty.
\end{align}
where $w=\sqrt{\exp^{(k+1)}(1)+|\omega|^2}$ and $\Phi(r)=\ln^{(k)}(|\ln r|)$.
\end{theorem}

\begin{remark}
Condition $1.$ is a description of the algebraic/polynomial nature of the blow-up profile (this is consistent with the current 3D NS theory)
suitable for our purposes; note that the degree is arbitrary, i.e., there is no restriction on the strength of the singularity. Condition $2.$
is a very mild (the number of $\log$-composites is arbitrary) condition on the oscillation of the vorticity direction allowing for a wide range of discontinuities.
\end{remark}


\begin{proof}
Recall that (cf. \citet{Bradshaw2015})
\begin{align*}
\partial_t w -\Delta w+(u\cdot\nabla)w\le \omega\cdot\nabla u\cdot \frac{\omega}{w}.
\end{align*}
Let $f(w)=\Phi(w)+w\Phi'(w)$. Then
\begin{align*}
\nabla f(w)= f'(w) \nabla w= \left(\Phi'(w)+(w\Phi'(w))'\right)\nabla w,
\end{align*}
and one can show that
$$0\le w\left(\Phi'(w)+(w\Phi'(w))'\right)\le 1$$
for all $w$; consequently, $wf'(w)$ is non-negative and bounded.
Multiplying the initial inequality by $\psi f$ yields
\begin{align*}
\psi f(w)\partial_t w -\psi f(w)\Delta w+\psi f(w)(u\cdot\nabla)w\le \psi f(w)\ \omega\cdot\nabla u\cdot \frac{\omega}{w}.
\end{align*}
Time-derivative:
\begin{align*}
\partial_t(\psi w \Phi(w))=&\psi\ \partial_t(w\Phi(w))
\\
=& \psi\left(\Phi(w)w_t + w\Phi'(w)w_t\right)
\\
=& \psi f(w) w_t.
\end{align*}
Laplacian:
\begin{align*}
-\Delta w \  \psi f(w)=&-\nabla\cdot\left(\psi f(w)\nabla w\right) + \nabla\left(\psi f(w)\right)\cdot \nabla w
\\
=&-\nabla\cdot\left(\psi f(w)\nabla w\right) + f(w)\nabla\psi\cdot \nabla w + \psi \nabla f(w)\cdot \nabla w
\\
=&-\nabla\cdot\left(\psi f(w)\nabla w\right) + f(w)\nabla\psi\cdot \nabla w + \psi f'(w)|\nabla w|^2.
\end{align*}
Note that after taking the integral $-\nabla\cdot\left(\psi f(w)\nabla w\right)$ vanishes, and since
$\psi f'(w)|\nabla w|^2$ is always non-negative, it can be dropped. \\
\\
Advection:
\begin{align*}
(u\cdot \nabla) w \  \psi f(w)=&\psi u\cdot \nabla F(w)= \nabla\cdot \left(\psi uF(w)\right) - \left(u\cdot\nabla \psi\right)F(w)
\end{align*}
where $F$ is an antiderivative of $f$ and $|F(w)|\lesssim |w|^{1+\epsilon}$ for arbitrary small $\epsilon$.
The first term vanishes after taking the integral.\\
\\
Vortex-stretching:
\begin{align*}
\omega\cdot \nabla u\cdot \frac{\omega}{w} \  \psi f(w) =& \omega\cdot \nabla u\cdot \xi \psi f(w) + \omega\cdot \nabla u\cdot \left(\frac{\omega}{w}-\frac{\omega}{|\omega|}\right) \psi f(w)
\\
=& \omega\cdot \nabla u\cdot \xi \psi \Phi(w) + \omega\cdot \nabla u\cdot \xi \psi w\Phi'(w) + \omega\cdot \nabla u\cdot \left(\frac{\omega}{w}-\frac{\omega}{|\omega|}\right) \psi f(w).
\end{align*}

In the spirit of  \cite{Bradshaw2015}, we arrive at
\begin{align*}
I(\tau)\equiv \int \psi(x)w(x,\tau)\Phi(w(x,\tau)) dx\le I(0)+c\int_0^\tau\int_x \omega\cdot \nabla u\cdot \psi\xi \Phi(w)dxdt + J(\tau)
\end{align*}
where
\begin{align*}
J=&-\int_0^\tau\int_x f(w)\nabla\psi\cdot \nabla w\ dxdt +\int_0^\tau\int_x (u\cdot \nabla\psi)F(w) dxdt
\\
&+\int_0^\tau\int_x \omega\cdot \nabla u\cdot \xi \psi w\Phi'(w) dxdt +\int_0^\tau\int_x \omega\cdot \nabla u\cdot \left(\frac{\omega}{w}-\frac{\omega}{|\omega|}\right) \psi f(w) dxdt
\\
=: & \ J_1+J_2+J_3+J_4.
\end{align*}

Note that $|f(w)|\lesssim w^\epsilon$ for arbitrary small $\epsilon$; H\"older inequality
with $1/p+1/q=1$ and \citet[Theorem 3.7]{Lions1996} with $L^q(\mathbb{R}^N\times (0,T))$ $(q<4/3)$ then yields
\begin{align*}
J_1\lesssim & \left(\int_t\int_x |wf'(w)+f(w)|^p\right)^{1/p} \left(\int_t\int_x |\nabla w|^q\right)^{1/q}
\lesssim  \left(\int_t\int_x w\right)^{1/p}<\infty.
\end{align*}

Turning to $J_2$, H\"older and Sobolev inequalities imply
\begin{align*}
J_2\lesssim & \int_0^\tau\|F(w)\nabla \psi\|_2 \|u\|_2 \lesssim\ \underset{t}{\sup}\|u\|_2 \int_0^\tau\|F(w)\nabla \psi\|_2
\\
\lesssim &\int_0^\tau\|F(w)\nabla \psi\|_2\lesssim \int_0^\tau\|\nabla(F(w)\nabla \psi)\|_{\frac{6}{5}}
\\
\lesssim &\int_0^\tau \|F(w)\|_{\frac{6}{5}}+\|f(w)\nabla w\|_{\frac{6}{5}}
\\
\lesssim &\int_0^\tau \|F(w)\|_{\frac{6}{5}}+\int_0^\tau \|f(w)\|_{p'}\|\nabla w\|_{q'}
\end{align*}
where $1/p'+1/q'=5/6$ and $6/5<q'<4/3$.
Note that
\begin{align*}
\|F(w)\|_{\frac{6}{5}}\lesssim \|w\|_2^{\frac{5}{3}} \ \mbox{and} \  \|f(w)\|_{p'}\lesssim \|w\|_1^{1/p'};
\end{align*}
hence, by using the result in \citet{Constantin1990} (also see \citet{Lions1996}),
\begin{align*}
J_2\lesssim & \int_0^\tau \|w\|_2^2+ \left(\underset{t}{\sup}\|w\|_1\right)^{1/p'}\int_0^\tau \|\nabla w\|_{q'}
\lesssim \int_t\int_x |\nabla u|^2+ \int_t\int_x |\nabla w|^{q'}<\infty.
\end{align*}

Regarding $J_3$ and $J_4$, it suffices to observe that
\begin{align}
|\xi w\Phi'(w)|\lesssim 1 \ \mbox{and} \  \left|\frac{\omega}{w}-\frac{\omega}{|\omega|}\right|\lesssim |\omega|^{-2},
\end{align}
respectively.

This completes the \emph{a priori} boundedness of $J$, and we can turn our attention to the remaining (leading order)
vortex-stretching term.

On one hand, by the convexity of $\ln^{(k)}x$ (for large $x$) we deduce
\begin{align*}
\left|\Phi(\cM_r w)-\Phi(w)\right|\le \left|\log(\cM_r w)-\log w\right|=\log(\cM_r w/w)\le \cM_r w/w -1;
\end{align*}
this yields
\begin{align*}
\int_0^\tau\int_x \omega\cdot \nabla u\cdot \psi\xi \left(\Phi(w)-\Phi(\cM_r w)\right)dxdt\le \int_0^\tau\int_x |\nabla u|\left(\cM_r w-\omega\right) dxdt
\end{align*}
which is \emph{a priori} bounded by H\"older inequality and the fact that $\cM_r$ is bounded on $L^2$. On the other hand, $\cH^1-BMO$ duality, Div-Curl Lemma (cf. \citet{Coifman1993}), Lemma~\ref{le:WBMO-Multiplier} and Condition $2.$ (\ref{eq:VortexWBMO}) combine to
\begin{align}
\int_0^\tau\int_x \omega\cdot \nabla u\cdot \psi\xi \Phi(\cM_r w)dxdt\lesssim &\int_0^\tau \|\omega\cdot \nabla u\|_{\cH^1} \|\psi\xi \Phi(\cM_r w)\|_{BMO}\ dt \notag
\\
\lesssim &\int_0^\tau \|\omega\cdot \nabla u\|_{\cH^1} \|\psi\xi \Phi(\cM_r w)\|_{\widetilde{bmo}}\ dt \notag
\\
\lesssim &\int_0^\tau \|\omega\|_2 \|\nabla u\|_2 \left(\|\psi\xi\|_\infty+\|\psi\xi\|_{\widetilde{bmo}_\varphi}\right) \|\Phi(\cM_r w)\|_{\widetilde{bmo}_\phi}\ dt \notag
\\
\lesssim &\ \underset{t\in (0,T)}{\sup}\|\Phi(\cM_r w)\|_{\widetilde{bmo}_\phi}\int_0^\tau\int_x |\nabla u|^2dxdt
\end{align}
where $\phi$ satisfies $\varphi=(\phi_*)^{-1}$ (here, we are referring to the notation of Lemma~\ref{le:WBMO-tDym}).
Consequently, Lemma~\ref{le:WBMO-tDym} paired with Condition $1.$ imply
\begin{align*}
\underset{t\in (0,T)}{\sup}\|\Phi(\cM_r w)\|_{\widetilde{bmo}_\phi}<\infty.
\end{align*}
This completes the proof.
\end{proof}

\begin{remark}
Instead of imposing a restriction on time dynamics of the vorticity near the possible singular time $T$, one can assume that the singular profile at $t=T$ has a somewhat regular \emph{shape}, e.g. an almost radial symmetry plus an almost monotonicity plus an almost convexity,
or satisfying the assumptions of Lemma~\ref{le:WBMO-SingPro} below, and assume that the mean oscillations leading to the blow-up time are not worse than the mean oscillations of the singular profile.
\end{remark}

\begin{appendix}

\section{Appendix}\label{sec:appendix}

\begin{lemma}\label{le:WBMO-SingPro}
Suppose a Lebesgue measurable function $f$ is compactly supported in a neighborhood of 0 and
\begin{enumerate}
\item $|f|\gtrsim |x|^{-\epsilon}$ for some $\epsilon$
\item $|\nabla f|\lesssim |x|^{-\ell}$ for some $\ell$.
\end{enumerate}
Then, $\phi_*(M_r f)\in\widetilde{bmo}_\phi$ where $\phi$ and $\phi_*$ are given in Lemma~\ref{le:WBMO-tDym}.
\end{lemma}

\begin{proof}
The proof is similar to Lemma~\ref{le:WBMO-tDym}, utilizing Lemma~\ref{le:Diff-MaxFun}.
\end{proof}

\bigskip

The rest of the Appendix is devoted to demonstrating that Condition $1.$ in Theorem~\ref{th:LogkEst-WBMO}
is essentially optimal within the mathematical framework utilized.

\begin{lemma}\label{le:DertoBMO}
Let $f:\mathbb R \to \mathbb C$ be differentiable such that
\begin{enumerate}
\item $f$ is even on $[-1,1]$
\item there exists a constant constant $C>0$ such that for every $y\in [-1,1]$ and $|x|\ge |y|$
\begin{equation}\label{e.convexity}
|f'(x)| \le C|f'(y)|
\end{equation}
\item $|xf'(x)| \lesssim \phi(x)$ for all $x\in [0,1]$.
\end{enumerate}
Then for every interval $I$ of length at most $1$ we have
$$\frac 1 {|I|} \int_I |f-f_I| \lesssim \phi(|I|).$$
In particular, if we assume that $f\in L^1$, then $f\in \widetilde{bmo}_{\phi}$.
\end{lemma}

\begin{proof}
Let $I=[a_I, b_I]$. Consider the case when $I\subset [0,\infty)$ first. Then by the triangle inequality
\begin{align*}
\frac 1 {|I|} \int_I |f(x)-f_I| dx &\le \frac 1 {|I|^2} \int_{I^2} |f(x)-f(y)| dx dy
\\
&=  \frac 2 {|I|^2} \iint_{a_I\le x \le y \le b_I} |f(x)-f(y)| dx dy
\\
&\le  \frac 2 {|I|^2} \iint_{a_I\le x \le y \le b_I} \int_x^y |f'(t)|dt dx dy
\\
&=\frac 2 {|I|^2}  \int_I |f'(t)| (t-a_I) (b_I-t) dt
\\
&\lesssim  \frac 1 {|I|}  \int_I |f'(t-a_I)| (t-a_I)  dt
\end{align*}
(using \eqref{e.convexity}  we have $|f'(t)| \lesssim |f'(t-a_I)|$, while clearly $0\le b_I-t\le |I|$)
\begin{align*}
= \frac 1 {|I|}  \int_0^{|I|} |f'(t)| t  dt \lesssim \frac 1 {|I|}  \int_0^{|I|} \phi(t) |I| dt \le \phi(|I|)
\end{align*}
using the monotonicity of $\phi$.

The case when $I\subset [-\infty,0]$ is entirely similar.

Now consider the remaining case when $a_I<0<b_I$. Without loss of generality we may assume that $|a_I| \le |b_I|$, and it suffices to show that for some constant $c$ we have
$$\frac 1 {|I|} \int_I |f-c| \lesssim \phi(|I|)$$
(This would imply  $(1/|I|) \int_I |f-f_I| =(1/|I|) \int_I |(f-c)-(f-c)_I| \le  (2/|I|) \int_I |f-c| \lesssim \phi(|I|)$.)

Set $c=\frac 1{b_I}\int_0^{b_I} f$. Since $|I|\le 1$, $I\subset [-1,1]$, and $f$ being even on $[-1,1]$ implies
$$\frac 1{|I|} \int_I |f-c| \lesssim \frac 1 {|I|} \int_0^{b_I} |f-c| \lesssim \frac 1 {b_I} \int_0^{b_I} |f-c|$$
which reduces this case to the case $I\subset [0,\infty)$ considered above.
\end{proof}

\bigskip

\begin{lemma}\label{le:BMOtoDer}
Let $f:  (0,1) \to (0,\infty)$  be differentiable, decreasing and integrable near $0$, and such that
for every open interval $I \subset (0,1)$ of sufficiently small length we have
$$\frac 1 {|I|} \int_I |f-f_I| \lesssim \phi(|I|).$$
\begin{enumerate}
 \item Assume that for $x$ sufficiently near $0$ $|xf'(x)| \gtrsim \psi(x)$ where   $\psi$ is essentially increasing nonnegative   on $(0,1]$. Then $\psi(x) \lesssim \phi(4x)$ for $x$ sufficiently near $0$.

 \item Assume that in addition $f$ is also near $0$.  Then $|xf'(x)| \lesssim \phi(x)$ near $0$.
\end{enumerate}
\end{lemma}

\proof
(1)  Let $I=(0,T)$ where $T<1$ is sufficiently small. By the mean value theorem $f(\alpha)=T^{-1} \int_0^T f(x)dx$ for some $\alpha\in (0,T)$, and since $f'(t)<0$ we have
\begin{align*}
T \phi(T) &\gtrsim \int_I  |f(x)-f_I|dx =  \int_0^\alpha  |f(x)-f(\alpha)|dx + \int_\alpha^T  |f(x)-f(\alpha)|dx
\\
&= \int_0^\alpha \int_x^\alpha |f'(t)|dt dx +  \int_\alpha^T \int_\alpha^x |f'(t)|dt dx
\\
&=\int_0^\alpha |f'(t)| \int_0^t dx dt + \int_\alpha^T |f'(t)| \int_t^T dx dt
\\
&=\int_0^\alpha |f'(t)| t dt + \int_\alpha^T |f'(t)| (T-t) dx dt= J_1+J_2.
\end{align*}
If $\alpha< T/2$, the second term can be bounded from below as
\begin{align*}
J_2 \gtrsim \int_{T/2}^{cT/2} |f'(t)| t dt \gtrsim \int_{T/2}^{cT/2} \psi(t)dt \gtrsim_c T\psi(T/2)
\end{align*}
using the essentially increasing property of $\psi$, and if $\alpha>T/2$, the first term as
\begin{align*}
J_1\gtrsim \int_{T/4}^{cT/4} \psi(t)dt \gtrsim T\psi(T/4).
\end{align*}
Consequently,
$$T\phi(T) \gtrsim \int_I |f(x)-f_I|dx  \gtrsim T\psi(T/4)$$
which implies the desired estimate.

(2) We proceed similarly as above and obtain (considering the two cases for $\alpha$)
\begin{align*}
T\phi(T) &\gtrsim \min\left\{\int_0^{T/2} |f'(t)| tdt, \int_{T/2}^T |f'(t)| (T-t)dt\right\}
\\
&\gtrsim T\min\left\{\int_{T/4}^{T/2} |f'(t)| dt, \int_{T/2}^{3T/4} |f'(t)| dt\right\}.
\end{align*}
Utilizing that $f'<0$ and $f'$ increasing near $0$ (because of convexity), we have that $|f'|$ is decreasing near $0$; hence,
for $T$ sufficiently small,
$$T\phi(T) \gtrsim  T^2\min\{|f'(T/2)|, |f'(3T/4)|\}\gtrsim T^2|f'(T)|.$$
Thus $x|f'(x)| \lesssim \phi(x)$ for $x$ near $0$ completing the proof.
\endproof

\bigskip

\begin{theorem}\label{th:NonCoifRoch}
There exists a sequence $\{f_n\}\subset L^p\cap C_c^\infty([-1,1])$ such that
$$\|\phi_*(M(f_n))\|_{\widetilde{bmo}_\phi}\to \infty$$
where $\phi_*(x)=\ln^{(k)}\left(|\ln x|\right)$, $k\ge1$ and $\phi(x)=-x\phi_*'(x)$ for small $x$.
\end{theorem}

\begin{remark}
It is worth noting that the above result fails when $k=0$, i.e., when $\phi_*(x)=|\ln x|$ for small $x$ and--consequently--$\phi = 1$ (this
follows directly from \citet{Coifman1980}). As a matter of fact, this is the main reason that the case $k \ge 1$ is more intricate
than the case $k=0$ considered in  \citet{Bradshaw2015}.
\end{remark}

\begin{proof}
Let us discuss the problem in one dimension (examples for higher dimensions can be obtained by rotating the one-dimensional
examples). Start with the function
\begin{align*}
f(x)=s^{-\delta}\cdot\chi_{[-s,s]}+\ell_1\cdot \chi_{[-s-\epsilon,-s)}+\ell_2\cdot \chi_{(s,s+\epsilon]}+C
\end{align*}
where $s, \delta,\epsilon>0$, $C>1$ and $\ell_1,\ell_2$ are two straight lines connecting the constant pieces in the continuous
way, and mollify the sharp corners at $\pm s, \pm (s+\epsilon)$ (we call the modified function $f$) in a way ensuring that $f\equiv s^{-\delta}+C$ on $[-s,s]$ and
$f\equiv C$ outside $[-\epsilon-s,s+\epsilon]$. Then for $x>s+\epsilon$ (or $x<-s-\epsilon$)
\begin{align*}
M(f)(x)= \frac{A+s^{-\delta}(1-\eta(x)/2)\eta(x)\epsilon}{x+\eta(x)\epsilon}+C
\end{align*}
where $A=\int_{-\epsilon-s}^{s+\epsilon} f\ dx$ and $\eta(x)=\frac{2(x-s)}{x+\sqrt{x^2+2\epsilon(x-s)}}$. Let
\begin{align*}
h=\phi_*(M(f)).
\end{align*}
In what follows, we will show that by sending the parameters $s$ and $\delta$ to $0$ in a suitable fashion, the weighted BMO norm of $h$ tends to infinity. More precisely, we are after the sequences $s_n$ and $\delta_n$ and some intervals $I_n=(ks_n,rs_n)$ (where $1\lesssim k,r\lesssim 1$) such that
\begin{align*}
|I_n|\ \underset{x\in I_n}{\inf}\ |h_n'(x)| / \phi(|I_n|^{-1}) \to \infty.
\end{align*}
Then Lemma~\ref{le:BMOtoDer}--(ii) would imply that $\|h_n\|_{\widetilde{bmo}_\phi}\to\infty$.
In order to apply Lemma~\ref{le:BMOtoDer} we will show that $h_n$ is monotone and
convex on $I_n$ and also analyze the asymptotic behavior of the quantity $|I_n| \underset{x\in I_n}{\inf} |h_n'(x)|/\phi(|I_n|^{-1})$.

A straightforward calculation gives
\begin{align*}
\frac{d}{dx}M(f)=&\ \frac{s^{-\delta}\epsilon(\gamma'-\gamma\gamma')(x+\gamma\epsilon)-(1+\gamma'\epsilon) (A+s^{-\delta}(1-\gamma/2)\gamma\epsilon)}{(x+\gamma \epsilon)^2}
\\
\frac{d^2}{dx^2}M(f)=&\ \epsilon\ F(\gamma,\gamma',\gamma'',x,\epsilon,A,s)+A(x+\gamma\epsilon)^{-4}
\end{align*}
where $\gamma'(x)=2-\frac{2(x-s)\left(1+(x+s)/\sqrt{x^2+2\epsilon(x-s)}\right)}{x+\sqrt{x^2+2\epsilon(x-s)}}$.
When $x$ is comparable to $s$ (i.e. $x\approx ks$), $\gamma'$ and $\gamma''$ are bounded. In addition, for a fixed $s$,
$F$ is bounded when $\epsilon\to0$. Hence, if $x$ is comparable to $s$ and $\epsilon$ is small enough ($\epsilon\ll s$), then
\begin{align*}
M(f)\approx Ax^{-1},\qquad \frac{d}{dx}M(f)\approx -Ax^{-2},\qquad \frac{d^2}{dx^2}M(f)\approx Ax^{-4};
\end{align*}
consequently for sufficiently small $s$ and $x$ ($x$ is comparable to $s$)
\begin{align*}
h' &=\phi_*'(M(f_n))\frac{d}{dx}M(f_n)\approx -\left(\ln Ax^{-1}\right)^{-1}x^{-1} <0,
\\
h'' &=\phi_*'(M(f_n))\frac{d^2}{dx^2}M(f_n) +\phi_*''(M(f_n))\left(\frac{d}{dx}M(f_n)\right)^2
\\
&\approx (\ln Ax^{-1})^{-1} x^{-3}- (\ln Ax^{-1})^{-2} x^{-2}>0.
\end{align*}
Now choose $k,r$ ($1\lesssim k,r\lesssim 1$) such that on $I=(ks,rs)$ the above requirements are satisfied (and Lemma~\ref{le:BMOtoDer} applies). For a fixed $\delta$, sending $s\to0$
\begin{align*}
|I|\ \underset{x\in I}{\inf}\ |h'(x)| / \phi(|I|^{-1})\approx s \left(\ln (s^{1-\delta}s^{-1})\right)^{-1}s^{-1}\big/\left(\ln s^{-1}\right)^{-1}\approx \delta^{-1}.
\end{align*}
Summarizing, for each $n$ we can choose $I_n$ such that $|I_n| \underset{x\in I_n}{\inf} |h_n'(x)|/ \phi(|I_n|^{-1})$ is close to $\delta_n^{-1}$
(with $\delta_n \to 0$) and thus
\begin{align*}
\underset{n}{\sup}\ |I_n| \underset{x\in I_n}{\inf} |h_n'(x)| / \phi(|I_n|^{-1}) =\infty.
\end{align*}
\end{proof}

\begin{corollary}
There exists a sequence $\{f_n\}\subset C_c^\infty([-1,1])$ such that $f_n\to f$ pointwise (also in $L^p$), $\phi_*(M_\delta(f))\in L^p\cap\widetilde{bmo}_\phi$ and $|f|\gtrsim |x|^{-\ell}$ but
$$\|\phi_*(M_\delta(f_n))\|_{\widetilde{bmo}_\phi}\to \infty$$
where $\phi_*(x)=\ln^{(k)}\left(|\ln x|\right)$ ($k\ge1$) and $\phi(x)=-x\phi_*'(x)$ for small $x$.
\end{corollary}

\begin{proof}
As in the previous theorem, it suffices to consider the one-dimensional setting. Consider the mollified version of the function
\begin{align*}
f(x)=s^{-\delta}\cdot\chi_{[-s,s]}+ t^{-\ell}\cdot\chi_{[-t,t]}+|x|^{-\ell}\cdot\chi_{\bR\setminus [-t,t]},\qquad t\gg s
\end{align*}
(and call it $f$) such that for $s+\epsilon<x<t$
\begin{align*}
M(f)= \frac{A+s^{-\delta}(1-\gamma(x)/2)\gamma(x)\epsilon}{x+\gamma(x)\epsilon}+t^{-\ell}
\end{align*}
where $\gamma$ is as in Theorem~\ref{th:NonCoifRoch} and $A\approx s^{1-\delta}$. The rest of the argument is very similar to that of Theorem~\ref{th:NonCoifRoch}, except that when sending $s$ and $\delta$ to 0 we let $s$ decrease much faster compared to $\delta$ because we want $Ax^{-1}$ to dominate $t^{-\ell}$. A suitable choice is given by $s=e^{-1/t}$ and $\delta=\sqrt{t}$. Then, if $x$ is comparable to $s$ (i.e. $x\approx ks$) we have
\begin{align*}
Ax^{-1}/ t^{-\ell}\approx t^{\ell}s^{-\delta}\approx t^{\ell}e^{1/\sqrt{t}}\to \infty
\end{align*}
and the same idea as in Theorem~\ref{th:NonCoifRoch} applies.
\end{proof}

\bigskip

\begin{corollary}
For any $\ell> n$, there exists a (infinite) collection $\{f_t\}\subset C_c^\infty([-1,1])$ such that
\begin{align*}
\|f_t\|_{L^1}\lesssim 1,\qquad |\nabla f_t|\lesssim |x|^{-\ell}
\end{align*}
but $\|\phi_*(M(f_t))\|_{\widetilde{bmo}_\phi}\to\infty$.
\end{corollary}

\begin{proof}
Consider a suitably mollified version of
\begin{align*}
f(x)=s^{-\delta}\cdot\chi_{[-s,s]}+\left(s^{-\delta}-s^{1-\ell}+|x|^{1-\ell}\right)\chi_{(s,b]} +\left(s^{-\delta}-s^{1-\ell}+|x|^{1-\ell}\right)\chi_{[-b,-s)}+C
\end{align*}
where $b$ is determined by $s^{-\delta}-s^{-\ell}+b^{-\ell}=0$. Note that if $\ell>1$ and $\delta$ is sufficiently small then $b-s\ll s$ and
$f\equiv C$ outside $[-b,b]$.  Similarly as in Theorem~\ref{th:NonCoifRoch}, for $x>s+b$ (or $x<-s-b$)
\begin{align*}
M(f)\approx Ax^{-1},\qquad \frac{d}{dx}M(f)\approx -Ax^{-2},\qquad \frac{d^2}{dx^2}M(f)\approx Ax^{-4}
\end{align*}
where $A=\int_{-b-s}^{s+b} f\ dx$. Setting $h=\phi_*(M(f))$, one can show $h$ is monotone and convex on some interval $I=(ks,rs)$ ($1\lesssim k,r\lesssim 1$) and
\begin{align*}
|I|\ \underset{x\in I}{\inf}\ |h'(x)| / \phi(|I|^{-1})\approx s \left(\ln \left(s^{1-\delta}s^{-1}\right)\right)^{-1}s^{-1}\big/\left(\ln s^{-1}\right)^{-1}\approx \delta^{-1}.
\end{align*}
Finally, we choose $f_{s(t),\delta(t)}$ with $s(t),\delta(t)\to0$ and $I_{s(t)}$ such that
\begin{align*}
|I_s|\underset{x\in I_s}{\inf} \left|h_{s(t),\delta(t)}'(x)\right| \big/ \phi(|I_s|^{-1}) \approx\delta^{-1}\to\infty.
\end{align*}
Lemma~\ref{le:BMOtoDer} then yields $\|\phi_*(M(f_{s(t),\delta(t)}))\|_{\widetilde{bmo}_\phi}\to\infty$.
In addition, note that--by construction--$\{f_{s,\delta}\}$ is bounded in $L^1$ and its gradients bounded by $|x|^{-\ell}$.

\end{proof}

\begin{theorem}
For any $\alpha<n$ and $0<\alpha<\beta$, there exists $f\in L^p\cap C^1(\bR^n\setminus\{0\})$ such that $|x|^{-\alpha}\lesssim|f|\lesssim |x|^{-\beta}$ but $$\phi_*(M_\delta(f))\notin \widetilde{bmo}_\phi.$$
\end{theorem}

\begin{proof}
For simplicity set $n=1$ and $\delta=1$. Let $f$ be a suitable mollification (on $(0, \infty)$) of
$$x^{-\alpha}\cdot\chi_{\bR\setminus\cup I_i}+\sum_ia_i^{-\beta}\cdot\chi_{I_i}$$
where $I_i=[a_i,b_i]$ are disjoint (e.g. $a_n=2^{-2^n}$) and $|I_i|\approx a_i^{\gamma_i}$ with $\gamma_i\to\infty$ and
$|I_i|\ll (a_i-b_{i+1})$. Then for $b_{i+1}\ll x<a_i$
\begin{align*}
(Mf)(x)\approx\frac{|I_i|a_i^{-\beta}+\int_x^{a_i}t^{-\alpha}dt}{b_i-x}
\end{align*}
(one can show that the average of $f$ on $[x,a_i]$ is increasing when $x\to a_i$ by differentiation).
A direct computation when $a_i-x\approx |I_i|$ then gives
\begin{align*}
\frac{d}{dx}(Mf)(x)&=\frac{-x^{-\alpha}(b_i-x)+(|I_i|a_i^{-\beta}+(1-\alpha)(a_i^{1-\alpha}-x^{1-\alpha}))} {(b_i-x)^2}
\\
&\approx \frac{-a_i^{-\alpha}(a_i-x+|I_i|)+(|I_i|a_i^{-\beta}+a_i^{-\alpha}(a_i-x))}{|I_i|^2} =\frac{a_i^{-\beta}-a_i^{-\alpha}}{|I_i|}.
\end{align*}
Let $g=\phi_*(M(f))$; then (for $a_i-x\approx |I_i|$)
\begin{align*}
|g'|&=|\phi(M(f))||M(f)|^{-1}\left|\frac{d}{dx}(Mf)(x)\right|
\\
&\gtrsim |\phi(a_i^{-\beta})|\cdot a_i^{\beta}\cdot a_i^{-\beta}|I_i|^{-1}\approx |\phi(a_i^{-\beta})||I_i|^{-1}.
\end{align*}
Consequently, for some small interval $J_i$ ($|J_i|\approx |I_i|$) which is close to $a_i$ we have
\begin{align*}
|J_i|\underset{x\in J_i}{\inf}|g'(x)| /\phi(|J_i|^{-1}) &\approx |J_i||\phi(a_i^{-\beta})||I_i|^{-1}/\phi(|J_i|^{-1})
\\
&\approx |\phi(a_i^{-\beta})|/\phi(|I_i|^{-1}) \approx |\phi(a_i^{-\beta})|/\phi(a_i^{-\gamma_i})
\\
&\gtrsim |\ln a_i^{-\gamma_i}|/|\ln a_i^{-\beta}|\approx \gamma_i/\beta\to\infty.
\end{align*}
Lemma~\ref{le:BMOtoDer} then implies $g=\phi_*(M(f))\notin \widetilde{bmo}_\phi$.
\end{proof}

\end{appendix}

\bigskip
\bigskip
\bigskip
\bigskip



\bibliographystyle{plainnat}
\def\cprime{$'$}


\end{document}